\documentclass{article}
\usepackage{amsmath,amsthm,amssymb,amscd}

\newcommand{\ga}{\alpha}
\newcommand{\gb}{\beta}
\renewcommand{\gg}{\gamma}
\newcommand{\gd}{\delta}
\newcommand{\gw}{\omega}

\newcommand{\gs}{\sigma}
\newcommand{\eps}{\varepsilon}

\newcommand{\cantor}{2^\gw}

\newcommand{\dotxgen}{{\dot x}_{\mathit{gen}}}

\newcommand{\supp}{\mathrm{supp}}

\newcommand{\dom}{\mathrm{dom}}

\newcommand{\coll}{\mathrm{Coll}}

\newtheorem{theorem}{Theorem}[section]

\newtheorem{claim}[theorem]{Claim}
\newtheorem{corollary}[theorem]{Corollary}

\newtheorem{proposition}[theorem]{Proposition}

\theoremstyle{definition}
\newtheorem{definition}[theorem]{Definition}

\title{Coloring triangles and rectangles\footnote{2020 AMS subject classification 03E35, 14P99, 05C15.}}

\author{
Jind{\v r}ich Zapletal\\
University of Florida\\
Institute of Mathematics, Czech Academy of Sciences\\
zapletal@ufl.edu}

\begin{document}
\maketitle

\begin{abstract}
It is consistent that ZF+DC holds, the hypergraph of rectangles on a given Euclidean space has countable chromatic number, while the hypergraph of equilateral triangles on $\mathbb{R}^2$ does not.
\end{abstract}

\section{Introduction}

This paper continues the study of algebraic hypergraphs on Euclidean spaces from the point of view of their chromatic numbers in the context of choiceless ZF+DC set theory. In the context of ZFC, such hypergraphs were completely classified by Schmerl regarding their countable chromatic number \cite{schmerl:avoidable}. In the choiceless context, the study becomes much more difficult and informative; in particular, the arity and dimension of the hypergraphs concerned begin to play much larger role. In this paper, I compare chromatic numbers of equilateral triangles with that of rectangles.

\begin{definition}
$\Delta$ denotes the hypergraph of arity three consisting of equilateral triangles on $\mathbb{R}^2$. Let $n\geq 2$ be a number. $\Gamma_n$ denotes the hypergraph of arity four consisting of rectangles on $\mathbb{R}^n$.
\end{definition}

\noindent In the base theory ZFC, these hypergraphs are well-understood. By an old result of \cite{ceder:basis}, $\Delta$ has countable chromatic number. On the other hand, for every number $n\geq 2$ the chromatic number of $\Gamma_n$ is countable if and only if the Continuum Hypothesis holds; this is a conjunction of \cite{erdos:old} and \cite[Theorem 2]{komjath:three}. In the base theory ZF+DC, I present an independence result:

\begin{theorem}
\label{maintheorem}
Let $n\geq 2$. It is consistent relative to an inaccessible cardinal that ZF+DC holds, the chromatic number of $\Gamma_n$ is countable, yet every non-null subset of $\mathbb{R}^2$ contains all vertices of an equilateral triangle.
\end{theorem}

\noindent Note that the conclusion implies that the chromatic number of $\Delta$ is not countable: in a partition of $\mathbb{R}^2$ into countably many pieces, not all of them can be Lebesgue null. The consistency result can be achieved simultaneously for all $n\geq 2$. The proof seems to use the algebra and geometry of both rectangles and equilateral triangles in a way which does not allow an easy generalization.

The paper follows the set theoretic standard of \cite{jech:newset}. The calculus of geometric set theory and balanced pairs in Suslin forcings is developed in \cite[Section 5.2]{z:geometric}. If $n>0$ is a natural number, $A\subset\mathbb{R}^n$ and $F\subset\mathbb{R}$ are sets, the set $A$ is \emph{algebraic over} $F$ if there is a polynomial $p(\bar x)$ with $n$ many variables and coefficients in $F$ such that $A$ is the set of all solutions to the equation $p(\bar x)=0$. The set $A$ is \emph{semialgebraic over} $F$ if there is a formula $\phi$ of real closed fields with parameters in $F$ and $n$ free variables such that $A=\{\bar x\in\mathbb{R}^n\colon \mathbb{R}\models\phi(\bar x)\}$. If $X, Y, C$ are sets and $C\subset X\times Y$ and $x\in X$ is an element, then $C_x$ is the vertical section of $C$ associated with $x$, $C_x=\{y\in Y\colon \langle x, y\rangle\in C\}$. Let $X$ be a Polish space and $\mu$ a $\gs$-finite Borel measure on it. If $M$ is a transitive model of ZFC and $x\in X$ is a point, then $x$ is \emph{random} over $M$ if it belongs to no $\mu$-null Borel subset of $X$ coded in $M$. By the Fubini theorem, for points $x_0, x_1\in X$ the following are equivalent: (a) $x_0$ is random over $M$ and $x_1$ is random over $M[x_0]$, (b) $x_1$ is random over $M$ and $x_0$ is random over $M[x_1]$, and (c) the pair $\langle x_0, x_1\rangle$ is random over $M$ for the product measure on $X\times X$. In each case, I will say that $x_0, x_1$ are mutually random over $M$. The only measure appearing in this paper is the Lebesgue measure on $\mathbb{R}^2$ and the word null always pertains to it. DC denotes the axiom of dependent choice.

The author was partially supported by grant EXPRO 20-31529X of GA \v CR.  

\section{A preservation theorem}

To prove Theorem~\ref{maintheorem}, it is necessary to isolate a suitable preservation property of Suslin forcing. First, recall the main concepts of \cite{z:noetherian}.

\begin{definition}
\textnormal{\cite[Section 2]{z:noetherian}} Let $X, Y$ be Polish spaces. 

\begin{enumerate}
\item A closed set $C\subset Y\times X$ is a \emph{Noetherian subbasis} if there is no infinite sequence $\langle a_n\colon n\in\gw\rangle$ of finite subsets of $Y$ such that the sets $D_n=\bigcap_{y\in a_n}C_y$ are strictly decreasing with respect to inclusion;
\item if $M$ is a transitive model of set theory containing the code for $C$ and $A\subset X$ is a set, let $C(M, A)=\bigcap\{C_y\colon y\in Y\cap M$ and $A\subset C_y\}$;
\item generic extensions $V[G_0]$ and $V[G_1]$ are \emph{mutually Noetherian} if for all Polish spaces $X, Y$ and Noetherian subbases $C\subset Y\times X$ coded in the ground model, if $A_0\subset X$ is a set in $V[G_0]$ then $C(V[G_1], A_0)=C(V, A_0)$, and if $A_1\subset X$ is a set in $V[G_1]$ then $C(V[G_0], A_1)=C(V, A_1)$.
\end{enumerate}
\end{definition}

\noindent For example, mutually generic extensions are mutually Noetherian \cite[Corollary 2.10]{z:noetherian}, and if $x_0, x_1$ are mutually random reals, then $V[x_0]$ and $V[x_1]$ are mutually Noetherian \cite[Corollary 3.14]{z:noetherian}. Another easy and important observation is that the Noetherian assumption on the subbasis $C$ implies that the intersection defining the set $C(M, A)$ is always equal to the intersection of a finite subcollection; therefore, the set $C(M, A)$ is coded in $M$ no matter whether $A\in M$ or not.

I will need the following strengthening of balance of Suslin forcings.

\begin{definition}
Let $P$ be a Suslin forcing. 

\begin{enumerate}
\item  A pair $\langle Q, \gs\rangle$ is $3, 2$-\emph{Noetherian balanced} if $Q\Vdash\gs\in P$ and  for every collection $\langle V[G_i]\colon i\in 3\rangle$ of pairwise mutually Noetherian extensions of $V$, every collection $\langle H_i\colon i\in 3\rangle$ of filters on $Q$-generic over $V$ in the respective models $V[G_i]$, every tuple $\langle p_i\colon i\in 3\rangle$ of conditions in $P$ stronger than $\gs/H_i$ in the respective models $V[G_i]$ has a common lower bound;
\item $P$ is $3, 2$-\emph{Noetherian balanced} if for every condition $p\in P$ there is a $3, 2$-Noetherian balanced pair $\langle Q, \gs\rangle$ such that $Q\Vdash\gs\leq\check p$.
\end{enumerate}
\end{definition}

\noindent The following theorem is stated using the parlance of \cite[Convention 1.7.18]{z:geometric}.

\begin{theorem}
\label{preservationtheorem}
In every forcing extension of the choiceless Solovay model by a cofinally $3, 2$-Noetherian balanced Suslin forcing, every non-null subset of $\mathbb{R}^2$ contains all vertices of an equilateral triangle.
\end{theorem}

\begin{proof}
Let $\kappa$ be an inaccessible cardinal. Let $P$ be a Suslin forcing which is $3, 2$-Noetherian balanced cofinally in $\kappa$. Let $W$ be a choiceless Solovay model derived from $\kappa$. Work in $W$. Suppose that $\tau$ is a $P$-name and $p\in P$ is a condition which forces $\tau$ to be a non-null subset of $\mathbb{R}^2$. I must find an equilateral triangle $\{x_0, x_1, x_2\}\subset\mathbb{R}^2$ and a condition in $P$ stronger than $p$ which forces $\check x_0, \check x_1, \check x_2\in\tau$. 

To do that, note that both $p, \tau$ are definable from a ground model parameter and an additional parameter $z\in \cantor$. Let $V[K]$ be an intermediate extension obtained by a poset of cardinality smaller than $\kappa$ such that $z\in V[K]$ and $P$ is $3, 2$-Noetherian balanced in $V_\kappa[K]$. Work in $V[K]$. Let $\langle Q, \gs\rangle$ be a $3, 2$-Noetherian balanced pair for the poset $P$ such that $|Q|<\kappa$ and $Q\Vdash\gs\leq\check p$. Let $R$ be the usual random poset of non-null closed subsets of $\mathbb{R}^2$ ordered by inclusion, adding an element $\dotxgen\in \mathbb{R}^2$. 

\begin{claim}
There is a poset $S$ of cardinality smaller than $\kappa$, a $Q\times R\times S$-name $\eta$ for a condition in $P$ stronger than $\gs$, and conditions $q\in Q$, $r\in R$, $s\in S$ such that  

$$\langle q, r, s\rangle\Vdash_{Q\times R\times S}\coll(\gw, <\kappa)\Vdash\eta\Vdash_P\dotxgen\in \tau.$$
\end{claim} 

\begin{proof}
Suppose towards a contradiction that this fails. In the model $W$, let $B\subset\mathbb{R}^2$ be the set of all points random over the model $V[K]$; thus, the complement of $B$ is null. Choose a filter $H_0\subset Q$ generic over $V[K]$ and consider the condition $\gs/H_0\leq p$ in the poset $P$. I will show that $\gs/H_0\Vdash_P\tau\cap B=0$, in contradiction to the initial assumptions of the condition $p\in P$.

To show this, let $x\in B$ be a point and $p_0\leq\gs/H_0$ be a condition; it will be enough to find a condition $p_1\in P$ compatible with $p_0$ which forces $\check x\notin\tau$. Let $H_1\subset Q$ be a filter generic over the model $V[K][H_0, x, p_0]$.
The contradictory assumption shows that $p_1=\gs/H_1\Vdash_P\check x\notin\tau$. At the same time, $V[K][H_0, x, p_0]$ and $V[K][H_1]$ are mutually generic extensions of the model $V[K]$. By the balance assumption on the pair $\langle Q, \gs\rangle$, the conditions $p_0$ and $p_1$ are compatible in $P$. This concludes the proof.
\end{proof}

\noindent Pick $S, \eta$ and $q\in Q, r\in R, s\in S$ as in the claim and move to the model $W$. Let $x_0\in r$ be a point random over $V[K]$. Since $x_0$ is a point of density of the set $r$, there must be a real number $\eps>0$ such that, writing $D\subset\mathbb{R}^2$ for the closed disc centered at $x_0$ of radius $\eps$, the relative measure of $r$ in $D$ is greater than $1/2$. Consider the measure-preserving self-homeomorphism $h$ of $\mathbb{R}^2$ rotating the whole plane around the point $x_0$ by angle $\pi/3$ counterclockwise. The disc $D$ is invariant under $h$; by the choice of $D$, $r\cap h^{-1}r\cap D$ is a closed set of positive measure. Let $x_1\in r\cap h^{-1}r\cap D$ be a point random over $V[K][x_0]$, and let $x_2=h(x_1)$. Clearly, the points $x_0, x_1, x_2\in r$ form an equilateral triangle. 

\begin{claim}
The points $x_0, x_1, x_2\in\mathbb{R}^2$ are pairwise random over $V[K]$.
\end{claim}

\begin{proof}
The point $x_1$ is chosen to be random over $V[K][x_0]$, therefore the points $x_0, x_1$ are mutually random over $V[K]$. The point $x_2$ is the image of $x_1$ under a measure-preserving self-homeomorphism in $V[K][x_0]$. Therefore, $x_2$ is random over $V[K][x_0]$, and $x_0, x_2$ are mutually random over $V[K]$. Finally, the point $x_2$ is the image of $x_0$ under the measure-preserving rotation around $x_1$ by angle $\pi/3$. Since $x_0$ is random over $V[K][x_1]$, so is $x_2$, and the points $x_1, x_2$ are mutually random over $V[K]$ as well.
\end{proof}

\noindent Now, for each $i\in 3$ let $H_i\subset Q$ and $G_i\subset S$ for $i\in 3$ be filters mutually generic over the model $V[K][x_j\colon j\in 3]$, containing the conditions $q, s$ respectively. The models $V[K][x_i]$ for $i\in 3$ are pairwise mutually Noetherian extensions of $V[K]$ by \cite[Corollary 3.14]{z:noetherian}. The models $V[K][x_i][G_i][H_i]$ for $i\in 3$ are then pairwise mutually Noetherian extensions of $V[K]$ as well by \cite[Proposition 2.9]{z:noetherian}. For each $i\in 3$ let $p_i=\eta/H_i, x_i, G_i\in P$. By the balance assumption on the pair $\langle Q, \gs\rangle$, the conditions $p_i$ for $i\in 3$ have a common lower bound in the poset $P$. By the forcing theorem applied in the respective models $V[K][x_i][G_i][H_i]$, this common lower bound forces $\{x_i\colon i\in 3\}\subset\tau$ as desired.
\end{proof}

\section{The coloring poset}

Let $n\geq 2$ be a number, and write $\Gamma_n$ for the hypergraph of rectangles in $\mathbb{R}^n$. To prove Theorem~\ref{maintheorem}, I must produce a $3, 2$-Noetherian balanced Suslin poset adding a total $\Gamma_n$-coloring. The definition of the poset uses, as a technical parameter, a Borel ideal $I$ on $\gw$ which contains all singletons and which is not generated by countably many sets. Further properties of the ideal $I$ seem to be irrelevant; the summable ideal will do.

\begin{definition}
\label{posetdefinition}
Let $n\geq 2$ be a number. The poset $P_n$ consists of partial functions $p\colon \mathbb{R}^n\to\gw$ such that there is a countable real closed subfield $\supp(p)\subset\mathbb{R}$ such that $\dom(p)=\supp(p)^n$, and $p$ is a $\Gamma_n$-coloring. The ordering is defined by $p_1\leq p_0$ if

\begin{enumerate}
\item $p_0\subset p_1$;
\item for every hypersphere $S\subset\mathbb{R}^n$ algebraic over $\supp(p_0)$ and any two points $x, y\in \dom(p_1\setminus p_0)$, if $x, y$ are opposite points on $S$ then $p_1(x)\neq p_1(y)$;
\item for any two parallel hyperplanes $P, Q\subset\mathbb{R}^n$ visible  in $\supp(p_0)$ and any two points $x, y\in \dom(p_1\setminus p_0)$, if $x, y$ are opposite points on the respective hyperplanes $P, Q$ then $p_1(x)\neq p_1(y)$;
\item if $a\subset\supp(p_1)$ is a finite set, then $p_1''\gd(p_0, p_1, a)\in I$ where $\gd(p_0, p_1, a)=\{x\in\dom(p_1\setminus p_0)\colon x$ is algebraic over $\supp(p_0)\cup a\}$.
\end{enumerate}
\end{definition}

\begin{proposition}
\label{closedproposition}
$\leq$ is a $\gs$-closed transitive relation.
\end{proposition}

\begin{proof}
For the transitivity, suppose that $r\leq q\leq p$ are conditions in the poset $P_n$; I must show that $r\leq p$. Checking the items of Definition~\ref{posetdefinition}, (1) is obvious. For (2), suppose that $S$ is a hypersphere algebraic over $p$ and $x, y$ are opposite points on it in $\dom(r\setminus p)$. By the closure properties of $\dom(q)$, either both $x, y$ belong to $\dom(q)$ or both do not. In the former case (2) is confirmed by an application of (2) of $q\leq p$, in the latter case (2) is confirmed by an application of (2) of $r\leq q$. (3) is verified in a similar way. For (4), suppose that $a\subset\supp(r)$ is a finite set. Let $b\subset\supp(q)$ be an inclusion maximal set of points algebraic over $\supp(p)\cup a$ which is algebraically independent. Since finite algebraically independent sets over  $\supp(p)$ form a matroid, it must be the case that $|b|\leq |a|$ holds. Note that $\gd(p, r, a)\subseteq \gd(p, q, b)\cup \gd(q, r, a)$ and $r''\gd(p,r,a)\subseteq q''\gd(p, q, b)\cup r''(q, r, a)$. Thus, the set $r''\gd(p, r, a)$ belongs to $I$, since it is covered by two sets which are in $I$ by an application of (4) of $q\leq p$ and $r\leq q$.

For the $\gs$-closure, let $\langle p_i\colon i\in\gw\rangle$ be a descending sequence of conditions in $P_n$, and let $q=\bigcup_i p_i$; I will show that $q$ is a common lower bound of the sequence. Let $i\in\gw$ be arbitrary and work to show $q\leq p_i$. For brevity, I deal only with item (4) of Definition~\ref{posetdefinition}. Let $a\subset\supp(q)$ be a finite set. There must be $j\in\gw$ greater than $i$ such that $a\subset\supp(p_j)$. By the closure properties of $\dom(p_j)$, it follows that $\gd(p_i, q, a)=\gd(p_i, p_j, a)$. Thus, $q''\gd(p_i, q, a)=p_j''\gd(p_i, p_j, a)$ and the latter set belongs to $I$ by an application of (4) of $p_j\leq p_i$.
\end{proof}

\noindent Further analysis of the poset $P_n$ depends on a characterization of compatibility of conditions.

\begin{proposition}
\label{compproposition}
Let $p_0, p_1\in P_n$ be conditions. The following are equivalent:

\begin{enumerate}
\item $p_0, p_1$ are compatible;
\item for every point $x_0\in\mathbb{R}^n$ there is a common lower bound of $p_0, p_1$ containing $x_0$ in its domain;
\item the conjunction of the following: 

\begin{enumerate}
\item $p_0\cup p_1$ is a function and a $\Gamma_n$-coloring;
\item whenever $S$ is a hypersphere visible from $\supp(p_0)$ and $x, y\in\dom(p_1\setminus p_0)$ are opposite points on $S$, then $p_1(x)\neq p_1(y)$;
\item whenever $P, Q$ are parallel hyperplanes visible from $\supp(p_0)$ and $x, y\in\dom(p_1\setminus p_0)$ are opposite points on them, then $p_1(x)\neq p_1(y)$;
\item for every finite set $a\subset\supp(p_1)$, $p_1''\gd(p_0, p_1, a)\in I$;
\item items above with subscripts $0, 1$ interchanged.
\end{enumerate}
\end{enumerate}
\end{proposition}

\begin{proof}
(2) implies (1), which in turn implies (3) by Definition~\ref{posetdefinition}. The hard implication is the remaining one: (3) implies (2). Suppose that all items in (3) obtain and $x_0\in\mathbb{R}^n$ is a point. To find a common lower bound of $p_0, p_1$ which contains $x_0$ in its domain, let $F\subset\mathbb{R}$ be a countable real closed field containing $x_0$ as an element and $\supp(p_0), \supp(p_1)$ as subsets. The common lower bound $q$ will be constructed in such a way that $\dom(q)=F^n$. Write $d= F^n\setminus (\dom(p_0)\cup \dom(p_1))$. For every point $x\in d$ and every $i\in 2$, let $\ga(x, i)=\{y\in\dom(p_i)\setminus\dom(p_{1-i})\colon y$ and $x$ are mutually algebraic over $\supp(p_{1-i})$.

\begin{claim}
For each $x\in d$ and $i\in 2$, the set $p''_i\ga(x, i)$ belongs to the ideal $I$.
\end{claim}

\begin{proof}
For definiteness set $i=1$. The set $\ga(x, 1)$ is a subset of $\gd(p_0, p_1, a)$ where $a$ is the set of coordinates of any point in $\ga(x, 1)$. The claim then follows from assumption (3)(d).
\end{proof}

\noindent Now, use the claim to find a set $b\subset\gw$ in the ideal $I$ which cannot be covered by finitely many elements of the form $p_i''\ga(x, i)$ for $x\in d$ and $i\in 2$ and finitely many singletons. Let $q\colon F^n\to\gw$ be a function extending $p_0\cup p_1$ such that $q\restriction d$ is an injection and for every $x\in d$, $q(x)\in b\setminus (p_0''\ga(x, 0)\cup p_1''\ga(x, 1))$. Such a function exists by the choice of the set $b$. I will show that $q\in P_n$ and $q$ is a lower bound of $p_0, p_1$.

To see that $q\in P_n$, let $R\subset\dom(q)$ be a rectangle and work to show that $R$ is not monochromatic. The treatment splits into cases.

\noindent \textbf{Case 1.} $R\subset\dom(p_0)\cup\dom(p_1)$. By the closure properties of the sets $\dom(p_0)$ and $\dom(p_1)$, there are two subcases.

\noindent\textbf{Case 1.1.} $R$ is entirely contained in one of the two conditions. Then $R$ is not monochromatic as both $p_0, p_1$ are $\Gamma_n$-colorings.

\noindent\textbf{Case 1.2.} There are exactly two vertices of $R$ in $\dom(p_0\setminus p_1)$ and exactly two vertices of $R$ in $\dom(p_1\setminus p_0)$. There are again two subcases.

\noindent\textbf{Case 1.2.1}  If the two vertices in $\dom(p_0\setminus p_1)$ are opposite on the rectangle $R$, then they determine a hypersphere visible from $\supp(p_0)$ on which the other two vertices are opposite as well. Then the other two vertices receive distinct $p_1$-colors by assumption (3)(b).

\noindent\textbf{Case 1.2.2.} If the two vertices in $\dom(p_0\setminus p_1)$ are next to each other on the rectangle $R$, then they determine parallel hyperplanes visible from $\supp(p_0)$ on which the other two vertices are opposite as well. Then the other two vertices receive distinct $p_1$-colors by assumption (3)(c).

\noindent \textbf{Case 2.} $R$ contains exactly one vertex in the set $d$; call it $x$. By the closure properties of the sets $\dom(p_0)$ and $\dom(p_1)$, the remaining three vertices of $R$ cannot all belong to the same condition. Thus, there must be two vertices contained in (say) $\dom(p_0)$ and one vertex, call it $y$, in $\dom(p_1\setminus p_0)$. Then $y, x$ are mutually algebraic over $\dom(p_0)$. Thus $y\in\ga(x, 1)$ and $q(x)\neq q(y)$ by the initial assumptions on the function $q$. In conclusion, the rectangle $R$ is not monochromatic.

\noindent\textbf{Case 3.} $R$ contains more than one vertex in the set $d$. Then $R$ is not monochromatic as $q\restriction d$ is an injection.

This shows that $q\in P_n$ holds. I must show that $q\leq p_1$; the proof of $q\leq p_0$ is symmetric. To verify Definition~\ref{posetdefinition} (2), suppose that $S$ is a hypersphere algebraic over $\dom(p_0)$ and $x, y\in\dom(q\setminus p_0)$ are opposite points on $S$. If $x, y\in\dom(p_1)$ then item (3)(b) shows that $q(x)\neq q(y)$. If $x\in d$ and $y\in\dom(p_0)$ (or vice versa) then $y\in \ga(x, 0)$ and $q(x)\neq q(y)$ by the choice of the color $q(x)$. Finally, if $x, y\in d$ then $q(x)\neq q(y)$ as $q\restriction d$ is an injection.

Definition~\ref{posetdefinition} (3) is verified in the same way.  For item (4) of Definition~\ref{posetdefinition}, let $a\subset F$ be a finite set. Let $a'\subset\supp(p_0)$ be a maximal set in $\supp(p_0)$ which is algebraically free over $\supp(p_1)$. Since algebraically free sets over $\supp(p_0)$ form a matroid, $|a'|\leq |a|$ holds, in particular $a'$ is finite. Now, $\gd(q, p_1, a)\subset\gd(p_1, p_0, a')\cup b$, the first set on the right belongs to $I$ by assumption (3)(d), so the whole union belongs to $I$ as required.
\end{proof}

\begin{corollary}
$P_n$ is a Suslin poset.
\end{corollary}

\begin{proof}
It is clear from Definition~\ref{posetdefinition} that the underlying set and the ordering of the poset $P_n$ are Borel. Proposition~\ref{compproposition} shows that the (in)compatibility relation is Borel as well.
\end{proof}

\begin{corollary}
\label{densecorollary}
$P_n$ forces the union of the generic filter to be a total $\Gamma_n$-coloring.
\end{corollary}

\begin{proof}
By a genericity argument, it is enough to show that for every condition $p\in P_n$ and every point $x_0\in\mathbb{R}^n$ there is a stronger condition containing $x_0$ in its domain. This follows from Proposition~\ref{compproposition} with $p=p_0=p_1$.
\end{proof}

\noindent It is time for the balance proof. It uses the following general fact.

\begin{proposition}
\label{subsetfact}
Let $V[G_0]$, $V[G_1]$ be mutually Noetherian extensions.

\begin{enumerate}
\item Let $n_0\in\gw$ be a number and $A\subset\mathbb{R}^{n_0}$ be a set algebraic over $V[G_1]$. Suppose that $\bar x_0\in V[G_0]\cap \mathbb{R}^{n_0}$ is a point in $A$. Then there is a set $B\subseteq A$ algebraic over $V$ such that $\bar x_0\in B$;
\item same as (1) except for semialgebraic sets;
\item if $a\subset \mathbb{R}\cap V[G_1]$ is a finite set and $r\in\mathbb{R}\cap V[G_1]$ is a real algebraic over $(\mathbb{R}\cap V[G_0])\cup a$, then $r$ is algebraic over $(\mathbb{R}\cap V)\cup a$;
\item $\mathbb{R}\cap V[G_0]\cap V[G_1]=\mathbb{R}\cap V$.
\end{enumerate}
\end{proposition}

\begin{proof}
For (1), let $n_1\in\gw$ be a number and $\phi(\bar v_0, \bar v_1)$ be a polynomial with integer coefficients and $n_0+n_1$ many free variables, and let $\bar x_1\in V[G_1]$ be an $n_1$-tuple of real numbers such that $A=\{\bar y\in\mathbb{R}^{n_0}\colon\phi(\bar y_, \bar x_1)=0\}$. Let $C\subset \mathbb{R}^{n_1}\times\mathbb{R}^{n_0}$ be the set of all pairs $\langle \bar y_1, \bar y_0\rangle$ such that $\phi(\bar y_0, \bar y_1)=0$. This is a Noetherian subbasis by the Hilbert Basis Theorem.
Since $C(V[G_1], \{\bar x_0\})\subseteq A=C_{\bar x_1}$ holds by the definitions and $C(V[G_1], \{\bar x_0\})=C(V, \{\bar x_0\})$ holds by the initial assumption on the generic extensions, (1) is witnessed by $B=C(V, \{\bar x_0\})$.

For (2), let $A\subset\mathbb{R}^{n_0}$ be a set semialgebraic over $V[G_1]$. By the elimination of quantifiers for real closed fields \cite[Section 3.3]{marker:book}, $A$ is defined by a quantifier free formula $\theta$ with parameters in $V[G_1]$. The formula can be rearranged so that its atomic subformulas compare a value of a polynomial with zero. Let $\{\phi_i\colon i\in m\}$ be a list of all polynomials mentioned in $\theta$. Let $\bar x_0\in V[G_0]\cap \mathbb{R}^{n_0}$ be a point in $A$. Let $a\subset m$ be the set of all $i$ such that $\phi_i(\bar x_0)=0$.
Let $O\subset\mathbb{R}^{n_0}$ be a rational open box around $\bar x_0$ in which the polynomials $\phi_i$ for $i\notin a$ do not change sign. Use (1) to find a set $C$ algebraic over $V$ such that $\bar x_0\in C$ and for all $\bar y\in C$ and all $i\in a$, $\phi_i(\bar y)=0$. It is clear that the set $B=C\cap O\subseteq A$ works as required in (2).

For (3), let $\phi(\bar x_0, \bar x_1, v)$ be a nonzero polynomial with parameters $\bar x_0$ in $V[G_0]$ and $\bar x_1\in a$ and a free variable $v$ such that $\phi(\bar x_0, \bar x_1, r)=0$ holds. There is an open neighborhood $O$ of $\bar x_0$ such that for every $\bar x'_0\in O$, the polynomial $\phi(\bar x'_0, \bar x_1, v)$ remains non-zero. Let $A=\{\bar y_0\colon \phi(\bar y_0, \bar x_1, r)=0\}$, use (1) to find a set $B\subseteq A$ algebraic over $V$ such that $\bar x_0\in B$, and use a Mostowski absoluteness argument to find a tuple $\bar x'_0\in O\cap B$ in the ground model. (3) is then witnessed by the tuple $x'_0$. Finally, (4) immediately follows from (1).
\end{proof}

\begin{theorem}
\label{dimensionbalance}
Let $n\geq 2$ be a number. In the poset $P_n$,

\begin{enumerate}
\item for every total $\Gamma_n$-coloring $c\colon \mathbb{R}^n\to\gw$, the pair $\langle \coll(\gw, \mathbb{R}), \check c\rangle$ is $3, 2$-Noetherian balanced;
\item if the Continuum Hypothesis holds then the poset $P_n$ is $3, 2$-Noetherian balanced.
\end{enumerate}
\end{theorem}

\noindent The fine details of this proof are the reason behind the rather mysterious Definition~\ref{posetdefinition}. 

\begin{proof}
For item (1), let $c\colon\mathbb{R}^n\to\gw$ be a total $\Gamma_n$-coloring. Let $V[G_i]$ for $i\in 3$ be pairwise mutually Noetherian extensions of $V$. Suppose that $p_i\leq c$ is a condition in $P_n$ in the model $V[G_i]$ for each $i\in 3$; I must find a common lower bound of all $p_i$ for $i\in 3$.

Work in the model $V[G_i\colon i\in 3]$. Let $F\subset\mathbb{R}$ be a countable real closed field containing $\supp(p_i)$ for $i\in 3$. I will construct a lower bound $q$ such that $F=\supp(q)$. Write $d=F^n\setminus\bigcup_i\dom(p_i)$. For each point $x\in d$ and for each pair $i, j\in 3$ of distinct indices, define sets $\ga(x, i, j), \gb(x, i, j)$ and $\gg(x, i, j)\subset\dom(p_i)$ as follows:

\begin{itemize}
\item $\ga(x, i, j)=\{y\in\dom(p_i\setminus c)\colon$ there is a hypersphere $S\subset\mathbb{R}^n$ algebraic over $\supp(p_j)$ such that $x, y$ are opposite points on $S\}$; 
\item $\gb(x, i, j)=\{y\in\dom(p_i\setminus c)\colon$ there are parallel hyperplanes $P, Q\subset\mathbb{R}^n$ algebraic over $\supp(p_j)$ such that $x, y$ are opposite points on $P, Q$ respectively$\}$;
\item $\gg(x, i, j)=\{y\in \dom(p_i\setminus c)\colon$ there are points $x_j\in \dom(p_j\setminus c)$ and $x_k\in\dom(p_k\setminus c)$ such that $x, y, x_j, x_k$ are four vertices of a rectangle listed in a clockwise or counterclockwise order$\}$. Here $k\in 3$ is the index distinct from $i$ and $j$.
\end{itemize}

\begin{claim}
There is a finite set $a\subset\supp(p_i)$ such that $\ga(x, i, j)$ consists of points algebraic over $(\mathbb{R}\cap V)\cup a$.
\end{claim}

\begin{proof}
This is clear if $\ga(x, i, j)=0$. Otherwise, let $y\in\ga(x, i, j)$ be any point and argue that all other points in $\ga(x, i, j)$ are algebraic over $(\mathbb{R}\cap V)\cup y$. To see this, suppose that $z\in\ga(x, i, j)$ is any other point. Let $S_y, S_z$ be hyperspheres algebraic in $\supp(p_j)$ such that $x$ is opposite of $y$ on $S_y$ and opposite of $z$ on $S_z$. It follows that $z$ is algebraic over $\supp(p_j)\cup y$: one first derives $x$ from $y$ and then $z$ from $x$. By Fact~\ref{subsetfact} $z$ is algebraic over $(\mathbb{R}\cap V)\cup y$ as desired.
\end{proof}

\begin{claim}
There is a finite set $a\subset\supp(p_i)$ such that $\gb(x, i, j)$ consists of points algebraic over $(\mathbb{R}\cap V)\cup a$.
\end{claim}

\begin{proof}
This is parallel to the previous argument.
\end{proof}

\begin{claim}
There is a finite set $a\subset\supp(p_i)$ such that $\gg(x, i, j)$ consists of points algebraic over $(\mathbb{R}\cap V)\cup a$.
\end{claim}

\begin{proof}
This is the heart of the whole construction and the reason why item (4) appears in Definition~\ref{posetdefinition}. For each point $y\in \gg(x, i, j)$ choose points $x_j(y)\in \dom(p_k\setminus c)$ and $x_k\in\dom(p_k\setminus c)$ witnessing the membership relation. Let $H(y)\subset\mathbb{R}^n$ be the hyperplane passing through $y$ and perpendicular to the vector $y-x_j(y)$; thus, $x\in H(y)$. Write $H=\bigcap_{y\in\gg(x, i, j)}H(y)$. Let $a\subset\gg(x, i, j)$ be a set of minimum cardinality such that $H=\bigcap_{y\in a}H(y)$; the set $a$ is finite. I will show that every point $y\in\gg(x, i, j)$ is algebraic over $(\mathbb{R}\cap V)\cup a$. This will prove the claim.

Let $y\in\gg(x, i, j)$ be an arbitrary point. Consider the set $A=\{u\in(\mathbb{R}^n)^{m+1}\colon \forall z\in\mathbb{R}^n\ (\forall l\in m\ (x_j(a(l))-u(l))\cdot (z-u(l))=0)\to (x_j(y)-u(m))\cdot (z-u(l))=0\}$. The set $A$ is semialgebraic in parameters from $\supp(p_j)$ and contains the tuple $a^\smallfrown y$. By Fact~\ref{subsetfact}(2) and the Noetherian assumption between $V[G_i]$ and $V[G_j]$, there is a set $B\subset A$ semialgebraic over $\mathbb{R}\cap V$ such that $a^\smallfrown y\in B$. Note that $B_a$ is a subset of the hypersphere of which the segment between $x_j(y)$ and $x$, and also the segment between $x_k(y)$ and $y$, is a diameter. Let $C=\{u\in B\colon u(m)$ is the farthest point of $B_{u\restriction m}$ from $x_i(y)\}$. This is a semialgebraic set in parameters from $\supp(p_k)$. By Fact~\ref{subsetfact}(2) and the Noetherian assumption on $V[G_i]$ and $V[G_k]$, there is a set $D\subseteq C$ semialgebraic over $\mathbb{R}\cap V$ such that $a^\smallfrown y\in D$. Clearly, $D_a=\{y\}$. It follows that $y$ is algebraic over $(\mathbb{R}\cap V)\cup a$ as desired.
\end{proof}

\noindent Now, define the set $f(x)\subset \gw$ of forbidden colors by setting it to the union of $p_i''(\ga(x, i, j)\cup \gb(x, i, j)\cup \gg(x, i, j))$ for all choices of distinct indices $i, j\in 3$.
By the claims and Definition~\ref{posetdefinition}(4) applied to $p_i\leq c$, $f(x)\in I$. Let $b\subset\gw$ be a set in the ideal $I$ which cannot be covered by finitely many sets of the form $f(x)$ for $x\in d$, and finitely many singletons. Let $q\colon F^n\to\gw$ be any map extending $\bigcup_ip_i$ and such that $q\restriction d$ is an injection such that $q(x)\in b\setminus f(x)$ holds for every $x\in d$. I claim that $q$ is the requested common lower bound of the conditions $p_i$ for $i\in 3$.

\begin{claim}
$q$ is a $\Gamma_n$-coloring.
\end{claim}

\begin{proof}
Let $R\subset F^n$ be a rectangle; I must show that $q$ is not constant on it. The proof breaks into numerous cases and subcases.

\noindent\textbf{Case 1.} $R$ contains no elements of the set $d$. Let $a\subset 3$ be an inclusion minimal set such that $R\subset\bigcup_{i\in a}\dom(p_i)$.

\noindent\textbf{Case 1.1.} $|a|=1$. Here, $R$ is not monochromatic because $p_i$ is a $\Gamma_n$-coloring where $i$ is the unique element of $a$.

\noindent\textbf{Case 1.2.} $|a|=2$, containing indices $i, j\in 3$. The closure properties of the domains of $p_i$ and $p_j$ imply that each set $\dom(p_i\setminus c)$ and $\dom(p_j\setminus c)$ contains exactly two points of $R$. 

\noindent\textbf{Case 1.2.1.} The two points in $\dom(p_i\setminus C)\cap R$ are adjacent in $R$. Then the hyperplanes containing the two respective points and perpendicular to their connector are algebraic over both $V[G_i]$ and $V[G_j]$, so in $V$ by Proposition~\ref{subsetfact}(4). The two points are opposite on these planes and therefore they receive distinct $p_i$ colors by Definition~\ref{posetdefinition}(3). Therefore, $R$ is not monochromatic.

\noindent\textbf{Case 1.2.2.} The two points in $\dom(p_i\setminus C)\cap R$ are opposite in $R$. Then both the center of the rectangle $R$ and the real number which is half of the length of the rectangle diagonal belong to both $V[G_i]$ and $V[G_j]$, so to $V$ by Proposition~\ref{subsetfact}(4). The hypersphere $S$ they determine is visible from $V$, and the two points of $\dom(p_i\setminus C)\cap R$ are opposite on $S$. Applying Definition~\ref{posetdefinition}(2) to $p_i\leq c$, it is clear that the two points receive distinct $p_i$ colors and $R$ is not monochromatic.

\noindent\textbf{Case 1.3.} $|a|=3$. Then there must be index $i\in 3$ such that $\dom(p_i\setminus c)$ contains exactly two points of $R$ and $\dom(p_j\setminus c)$ contains exactly one point of $R$ for each index $j\neq i$. I will show that this case cannot occur regardless of the colors on the rectangle $R$. For an index $j\neq i$, write $x_j$ for the unique point in $R\cap\dom(p_j\setminus c)$.

\noindent\textbf{Case 1.3.1.} The two points in $\dom(p_i\setminus C)\cap R$ are adjacent in $R$. Consider the two hyperplanes $Q_j, Q_k$ containing these two points respectively and perpendicular to their connecting segment, indexed by $j, k\neq i$. Reindexing if necessary, $x_j\in Q_j$ and $x_k\in Q_k$ holds. By Proposition~\ref{subsetfact}(1), there must be algebraic sets $Q'_j\subseteq Q_j$ and $Q'_k\subseteq Q_k$ visible from the ground model and still containing $x_j$ and $x_k$. This means that $x_k$ can be recovered in $V[G_j]$ as the closest point to $x_j$ in $Q'_k$. This is impossible as $\mathbb{R}\cap V[G_j]\cap V[G_k]=\mathbb{R}\cap V$.

\noindent\textbf{Case 1.3.2.} The two points in $\dom(p_i\setminus C)\cap R$ are opposite in $R$. Consider the hypersphere $S$ in which these two points are opposite. $S$ then contains $x_j$ and $x_k$ and these two points are opposite in $S$. By Fact~\ref{subsetfact}, there must be algebraic sets $S_j\subseteq S$ and $S_k\subseteq S$ visible from the ground model and still containing $x_j$ and $x_k$. This means that $x_k$ can be recovered in $V[G_j]$ as the farthest point to $x_j$ in $S_k$. This is impossible as $\mathbb{R}\cap V[G_j]\cap V[G_k]=\mathbb{R}\cap V$.

\noindent\textbf{Case 2.} $R$ contains exactly one point in the set $d$; call this unique point $x$. Let $a\subset 3$ be an inclusion minimal set such that $R\setminus\{x\}\subset\bigcup_{i\in a}\dom(p_i)$.

\noindent\textbf{Case 2.1.} $|a|=1$. This cannot occur since $\dom(p_i)$ would contain $x$ with the other three vertices of $R$, where $i\in 3$ is the only element of $a$.

\noindent\textbf{Case 2.2.} $|a|=2$, containing indices $i, j\in 3$. Here, for one of the indices (say $j$) $\dom(p_j)$ has to contain two elements of $R$ while $\dom(p_i\setminus c)$ contains just one; denote the latter point by $x_i$.

\noindent\textbf{Case 2.2.1.} The points $x_i$ and $x$ are opposite on the rectangle $R$. Then $x_i\in\ga(x, i, j)$ as the hypersphere on which $x_i, x$ are opposite points is the same as the one on which the other two points are opposite, and therefore is algebraic over $\supp(p_j)$. The choice of the map $q$ shows that $q(x)\neq p_i(x_i)$, so $R$ is not monochromatic.

\noindent\textbf{Case 2.2.1.} The points $x_i$ and $x$ are opposite on the rectangle $R$. Then $x_i\in\gb(x, i, j)$ as $x_i, x$ are opposite points on the hyperplanes passing through the other two points and perpendicular to their connecting segment, and these are algebraic over $\supp(p_j)$. The choice of the map $q$ shows that $q(x)\neq p_i(x_i)$, so $R$ is not monochromatic.

\noindent\textbf{Case 2.3.} $|a|=3$. For each index $i\in 3$ let $x_i\in R$ be the unique point in $\dom(p_i\setminus c)$. Let $i, j, k\in 3$ be indices such that the sequence $x, x_i, x_j, x_k$ goes around the rectangle $R$. Then $x_i\in \gg(x, i, j)$ holds. The choice of the map $q$ shows that $q(x)\neq p_i(x_i)$, so $R$ is not monochromatic.

\noindent\textbf{Case 3.} $R$ contains more than one point in the set $d$. Then $R$ is not monochromatic as $d\restriction d$ is an injection.
\end{proof}

\noindent Finally, let $i\in 3$ be an index; I must prove that $q\leq p_i$ holds. It is clear that $p_i\subset q$ holds. The following claims verify the other items of Definition~\ref{posetdefinition}. 

\begin{claim}
\label{sphereclaim}
If $S\subset\mathbb{R}^n$ is a hypersphere algebraic over $\supp(p_i)$ and $x, y\in\dom(q\setminus p_i)$ are opposite points on it, then $q(x)\neq q(y)$.
\end{claim}

\begin{proof}
The arguments splits into cases.

\noindent\textbf{Case 1.} If $x, y$ both belong to the set $d$, then $q(x)\neq q(y)$ as $q\restriction d$ is an injection.

\noindent\textbf{Case 2.} If $x\in d$ and $y\notin d$, let $j\in 3$ be an index distinct from $i$ such that $y\in\dom(p_j\setminus c)$. Then, $y\in\ga(x, j, i)$ holds and therefore $q(x)\neq p_j(y)$ as $q(x)\notin p_j''\ga(x, j, i)$.

\noindent\textbf{Case 3.} If neither of the points $x, y$ belongs to $d$, then there are two subcases.

\noindent\textbf{Case 3.1.} There is $j\in 3$ such that both $x, y$ belong to $\dom(p_j)\setminus c$. In such a case, the hypersphere $S$ is also algebraic over $\supp(p_j)$. By the Noetherian assumption on the models $V[G_i]$ and $V[G_j]$ and Proposition~\ref{subsetfact}(4), the hypersphere $S$ is algebraic over the ground model. It follows that $q(x)=p_j(x)\neq p_j(y)=q(y)$ by Definition~\ref{posetdefinition} (2) applied to $p_j\leq c$.

\noindent\textbf{Case 3.2.} $x\in \dom(p_j\setminus c)$ and $y\in\dom(p_k\setminus c)$ for distinct indices $j, k$. By the Noetherian assumption on the models $V[G_i]$ and $V[G_j]$ and Proposition~\ref{subsetfact}(1), there is a set $T\subseteq S$ algebraic over the ground model such that $x\in T$. Then $x$ can be recovered in $V[G_k]$ as the point on $T$ farthest away from $y$, contradicting the fact that $V[G_j]\cap V[G_k]=0$.
\end{proof}

\begin{claim}
If $P, Q\subset\mathbb{R}^n$ are parallel hyperplanes algebraic over $\supp(p_i)$ and $x, y\in\dom(q\setminus p_i)$ are opposite points on them, then $q(x)\neq q(y)$.
\end{claim}

\begin{proof}
The argument is similar to that for Claim~\ref{sphereclaim}.
\end{proof}

\begin{claim}
If $a\subset F^n$ is a finite set, then $q''\gd(p_i, q, a)\in I$. 
\end{claim}

\begin{proof}
For each index $j\in 3$ distinct from $i$, let $a_j\subset \gd(p_i, q, a)\cap \dom(p_j)$ be an inclusion-maximal set which is algebraically free over $\supp(p_i)$. Since sets algebraically free over $\supp(p_i)$ form a matroid, $|a_j|\leq |a|$. By Proposition~\ref{subsetfact}(3) $\gd(p_i, q, a)\cap \dom(p_j)=\gd(c, p_j, a_j)$ holds. This means that $\gd(p_i, q, a)=\gd(c, p_j, a_j)\cup \gd(c, p_k, a_k)\cup d$ where $j, k\in 3$ are the two indices distinct from $i$. Now, $p_j''(c, p_j, a_j)\in I$ by Definition~\ref{posetdefinition}(4) applied to $p_j\leq c$, $p_k''(c, p_k, a_j)\in I$ by Definition~\ref{posetdefinition}(4) applied to $p_k\leq c$, and $q''b\in I$ as this set is a subset of $b$. As the ideal $I$ is closed under unions and subsets, $q''\gd(p_i, q, a)\in I$ as desired.
\end{proof}

\noindent This concludes the proof of item (1) of the theorem. For item (2), if CH holds and $p\in P_n$ is a condition, by (1) it is enough to produce a total $\Gamma_n$-coloring $c$ such that $\coll(\gw, \mathbb{R})\Vdash\check c\leq\check p$. To this end, choose an enumeration $\langle x_\ga\colon\ga\in\gw_1\rangle$ of $\mathbb{R}^n$ and by recursion on $\ga\in\gw_1$ build conditions $p_\ga\in P_n$ so that

\begin{itemize}
\item $p=p_0\geq p_1\geq\dots$;
\item $x_\ga\in\dom(p_{\ga+1})$;
\item $p_{\ga}=\bigcup_{\gb\in\ga}p_\gb$ for limit ordinals $\ga$.
\end{itemize}

\noindent The successor step is possible by Corollary~\ref{densecorollary} and the limit step by Proposition~\ref{closedproposition}. In the end, let $c=\bigcup_\ga p_\ga$ and observe that $c$ is a total $\Gamma_n$-coloring and $c\leq p$.
\end{proof}

\noindent Finally, I can complete the proof of Theorem~\ref{maintheorem}. Let $n\geq 2$ be a number. Let $\kappa$ be an inaccessible cardinal. Let $W$ be the choiceless Solovay model derived from $\kappa$. Let $P_n$ be the Suslin poset of Definition~\ref{posetdefinition}, and let $G\subset P_n$ be a filter generic over $W$. $W[G]$ is a model of ZF+DC since it is a $\gs$-closed extension of a model of ZF+DC. In $W[G]$, the chromatic number of $\Gamma_n$ by Corollary~\ref{densecorollary}. In $W[G]$, every non-null subset of the plane contains an equilateral triangle by the conjunction of Theorem~\ref{preservationtheorem} and Theorem~\ref{dimensionbalance}. The proof is complete.

\bibliographystyle{plain} 
\bibliography{odkazy,zapletal,shelah}

\end{document}